\documentclass[12pt,english]{article}
\usepackage[T1]{fontenc}
\usepackage[latin9]{inputenc}
\usepackage{amsmath}
\usepackage{amsthm}
\usepackage{amssymb}
\usepackage{enumerate}   
\usepackage{color}
\numberwithin{equation}{section}
\makeatletter

\theoremstyle{plain}
\theoremstyle{remark}

\usepackage{microtype}
\usepackage{fullpage}
\usepackage{makeidx}
\usepackage{amsfonts}
\usepackage{latexsym}
\usepackage{babel}
\usepackage{babel}
\makeatother

\def \1{{\bf 1}}
\def\C{{\mathbb{C}}}
\def\Z{{\mathbb{Z}}}

\def\N{{\mathbb{N}}}

\theoremstyle{definition}
\newtheorem{lemma}{Lemma}[section]
\newtheorem{theorem}[lemma]{Theorem}
\newtheorem{corollary}[lemma]{Corollary}
\newtheorem{proposition}[lemma]{Proposition}
\newtheorem{definition}[lemma]{Definition}

\newtheorem{example}[lemma]{Example}
\newtheorem{remark}[lemma]{Remark}

\usepackage{times}

\usepackage[blocks]{authblk}

\title{$Y(z)$-injective vertex superalgebras and Hopf actions}

\author{Chao Yang  \footnote{supported by NSFC (No. 12301039), 
		the Fundamental Research Funds for the Central Universities (No. 2682024CX012), 
		and the Natural Science Foundation of Sichuan Province  (No. 2025ZNSFSC0797)} }
\affil{School of Mathematics,  Southwest Jiaotong University,
	Chengdu 611756, China }

\begin{document}
	
\maketitle

\abstract

This paper investigates \( Y(z) \)-injective vertex superalgebras. 
We first establish that two fundamental classes of vertex superalgebras\text{---}simple ones and those admitting a PBW basis\text{---}are \( Y(z) \)-injective. 
We then study actions of Hopf algebras on \( Y(z) \)-injective vertex superalgebras and prove that
every finite-dimensional Hopf algebra acting inner faithfully on such algebras must be a group algebra.
As a direct consequence, the study of the structure and representation theory of fixed-point subalgebras under 
finite-dimensional Hopf algebra actions reduces to that under group actions.

\section{Introduction}

To unify the study of group actions and Lie algebra actions on vertex operator algebras, 
a notion of Hopf algebra actions was introduced in \cite{DW}.
Given a vertex operator algebra $V$ with an action of a Hopf algebra $H$,  the fixed point subspace $V^H$ is also a vertex operator algebra.
Two central problems arise in this context:
1) Determine what types of Hopf algebras can act on a vertex operator algebra;
2) Understand  the structure and representation theory of $V^H$. 
 
In \cite{DW}, it was established that any finite-dimensional 
Hopf algebra admitting a faithful action on a simple vertex operator algebra is necessarily a group algebra.
Building on this foundation, recent work in \cite{DRY2} generalizes the results of \cite{DW} to  the setting of vertex algebras.
It proves that any finite-dimensional Hopf algebra acting inner faithfully on a $Y(z)$-injective vertex algebra must be a group algebra.
We note that the term "$\pi_2$-injective vertex algebra"  used in \cite{DRY2} is  referred to as "$Y(z)$-injective vertex algebra" in this paper.
The primary objective of this paper is to extend  the results of \cite{DRY2} to vertex superalgebras.

A vertex (super)algebra $V$ is said to be \emph{$Y(z)$-injective} if the linear map
\[
Y(z) \colon V \otimes V \to V(\!(z)\!) , \quad 
u \otimes v \mapsto Y(u, z)v \quad \text{for } u, v \in V,
\]
is injective. Vertex (super)algebras satisfying this $Y(z)$-injectivity condition 
play crucial roles in orbifold theory and the study of Hopf algebra actions on vertex (super)algebras \cite{ALPY1,CRY, DM,DRY1,DRY2,DY,T}. 
In this work, we investigate $Y(z)$-injective vertex superalgebras.

We first establish $Y(z)$-injectivity for two fundamental classes of vertex superalgebras:
simple vertex superalgebras and those with a PBW basis. 
For simple vertex superalgebras, we adapt the method from \cite{DRY2}. 
The key new ingredient is Lemma \ref{simple-AVD}, 
which asserts that a simple vertex superalgebra $V$ 
remains simple as an $A(V, \mathcal{D})$-module. 
This preservation of simplicity is non-trivial for vertex superalgebras 
because they can possess non-trivial nonhomogeneous ideals.

For vertex superalgebras with a PBW basis, we adapt an argument from \cite{Li1}---originally 
developed to prove nondegeneracy of such vertex algebras---to establish their $Y(z)$-injectivity, specifically by showing that: (1) for such $V$, the filtered commutative vertex superalgebra $\operatorname{gr}_E(V)$ is $Y(z)$-injective [Theorem \ref{main-Y(z)}], and (2) this $Y(z)$-injectivity of $\operatorname{gr}_E(V)$ implies that of $V$ [Lemma \ref{Yz-1}].

Finally, we investigate what kinds of Hopf algebras can act inner faithfully on a $Y(z)$-injective vertex algebra. 
Using arguments analogous to those for vertex algebras in \cite{DRY2},
we show that if a finite-dimensional Hopf algebra acts inner-faithfully on a $Y(z)$-injective
vertex superalgebra, then it must be a group algebra. 
As a consequence, the structure and representation theory of fixed-point subalgebras under finite-dimensional Hopf actions reduces to that of group actions.

This paper is organized as follows: In Section 2, we review foundational concepts and key examples of vertex superalgebras.
In Section 3, we show that all simple vertex superalgebras are $Y(z)$-injective. 
In Section 4, we prove that vertex superalgebras admitting a PBW basis are $Y(z)$-injective. 
In Section 5, we prove that every finite-dimensional Hopf algebra acting inner faithfully on a $Y(z)$-injective vertex superalgebra must be a group algebra.

{\bf Conventions}: Throughout this paper, we work over the complex field $\C$.
The unadorned symbol $\otimes$ means the tensor product over $\C$. 
We denote by $\N$ the set of nonnegative integers.
$\Z_2=\{\bar{0},  \bar{1} \}$ denotes the cyclic group of order $2$.

\section{Preliminaries}

\subsection{Vertex superalgebras}

A  {\em vector superspace} is a vector space $V$ with a  $\mathbb{Z}_{2}$-grading $V=V_{\bar{0}} \oplus V_{\bar{1}}.$
An element $u$ in $V$ is said to be {\em homogeneous} if it belongs to either $ V_{\bar 0} $ or $V_{\bar 1}$.
The elements of $V_{\bar 0}$ (resp., $V_{\bar 1}$) are called {\em even} (resp., {\em odd}).
If $u \in V_{\bar i}$ for $ i \in \{0, 1\}$, we write $|u|=i$.

Let $V$ be a vector superspace. The {\em canonical linear automorphism} $\sigma_V \colon V \to V$ is defined by
$\sigma_V(u)=(-1)^{|u|}u$ for any homogeneous element $u \in V$.
For any subspace $W$ of $V$, define $W_{\bar 0}=W \cap V_{\bar 0}$ and $W_{\bar 1}=W \cap V_{\bar 1}.$
A subspace $W$ of $V$ is called a \emph{homogeneous subspace} (or \emph{subsuperspace}) if it can be decomposed as $W = W_{\bar{0}} \oplus W_{\bar{1}}$.
Equivalently, $W$ is homogeneous if and only if it is stable under $\sigma_V$ (i.e., $\sigma_V(W) = W$).

For a vector superspace $V$, let $|V|$ denote the underlying
vector space obtained by forgetting its $\Z_2$-grading.

\begin{definition}
	A vertex superalgebra is a triple \((V, Y(\ , z), \mathbf{1})\) consisting of:
	\begin{itemize}
		\item A vector superspace \(V = V_{\bar{0}} \oplus V_{\bar{1}}\),
		\item The vacuum vector \(\mathbf{1} \in V_{\bar{0}}\),
		\item A linear map:
		\[	Y(\ , z) : V \to \text{End}_{\mathbb{C}}(V)[[z, z^{-1}]], \quad \ \
		v \mapsto Y(v, z)= \sum_{n \in \mathbb{Z}} v_n z^{-n-1} \]
	\end{itemize}
	satisfying:
\begin{enumerate}[{(1)}]
	\item Given $u, v \in V$, we have $u_nv=0$ \ for $n \gg 0$.
	
	\item $Y({\bf 1},z)=id_{V}$, \  and \   $Y(v,z){\bf 1}= v +(v_{-2}{\bf 1})z +  \cdots \in V[[z]]$.
	
	\item If $u \in V_{\alpha}$ and $v \in V_{\beta}$,  then  $u_nv \in V_{\alpha+\beta}$  \ for any $\alpha, \beta \in \Z_2$ and any $n \in \Z$.
	
	\item The following Jacobi identity holds for any homogeneous $u , v ,w  \in V$:
	\begin{align*}
		& \displaystyle{ z^{-1}_0\delta\left(\frac{z_1-z_2}{z_0}\right)  Y(u,z_1)Y(v,z_2)w
			-(-1)^{|u||v|}z^{-1}_0\delta\left(\frac{z_2-z_1}{-z_0}\right) Y(v, z_2)Y(u,z_1)}w\\
		& \displaystyle{=z_2^{-1}\delta\left(\frac{z_1-z_0}{z_2}\right) Y(Y(u,z_0)v,z_2)}w.
	\end{align*}
\end{enumerate}
\end{definition}

\begin{definition}
	Let $T$ be a positive integer. A \emph{$\frac{1}{T}\mathbb{N}$-graded vertex superalgebra} is a vertex superalgebra $V$ equipped with a $\frac{1}{T}\mathbb{N}$-grading
	\[
	V = \bigoplus_{n \in \frac{1}{T}\mathbb{N}} V_n 
	\]
	satisfying the following conditions:
	\begin{enumerate}[(1)]
		\item $\mathbf{1} \in V_0$;
		\item $V_\alpha = \bigoplus\limits_{n \in \frac{1}{T}\mathbb{N}} (V_\alpha \cap V_n)$ for each $\alpha \in \Z_2$;
		\item $u_sV_{n} \subseteq V_{n+m-s-1}$ for any $u \in V_m, s \in \Z$, and $ m, n \in \frac{1}{T}\N$.
	\end{enumerate}
If $v \in V_n$ for $n \in \frac{1}{T}\mathbb{N}$, write $\deg v = n$.
An element $u \in V$ is \emph{$(\mathbb{Z}_2 \times \frac{1}{T}\mathbb{N})$-homogeneous} 
if $u \in V_\alpha \cap V_n$ for some $\alpha \in \mathbb{Z}_2$, $n \in \frac{1}{T}\mathbb{N}$.

\end{definition}

\begin{definition}
	Let $V$ be a vertex superalgebra, and let $U\subset V$ be a  subset. $V$ is said to be strongly generated by $U$ if $V$ is spanned by elements of the form:
	$$u^{1}_{-n_{1}}...u^{r}_{-n_{r}}{\bf 1}, $$
	where $r \geq 0$, $u^{1},...,u^{r} \in U$, and $n_{i}\geq 1$ for all $i$.
\end{definition}

For a vertex superalgebra \(V\), let \(\mathcal{D}\) be the even linear map 
\(\mathcal{D} \colon V \to V\) defined by \(\mathcal{D}(v) = v_{-2}\mathbf{1}\) for \(v \in V\).

\begin{proposition}[\cite{LL}]\label{Der}
	The following identities hold for homogeneous elements \(u, v \in V\):
	\begin{enumerate}[(1)]
		\item \(Y(\mathcal{D}v, z) = \frac{d}{dz}Y(v, z)\);
		\item \(Y(u, z)v = (-1)^{|u||v|} e^{z\mathcal{D}} Y(v, -z)u\).
	\end{enumerate}
\end{proposition}

\begin{definition}

An \emph{automorphism} of a vertex superalgebra $V$ is an {\bf even} invertible linear map $g : V \to V$ 
satisfying $g(\mathbf{1}) = \mathbf{1}$ and 
$g Y(u, z) g^{-1} = Y(gu, z)$ for all $u \in V$. 
The set of all automorphisms of $V$ is denoted $\operatorname{Aut}(V)$. 
Note that the canonical automorphism $\sigma_V$ lies in the center of $\operatorname{Aut}(V)$.

\end{definition}

\begin{definition}  \label{main-def}

Let $V$ be a vertex superalgebra.

\begin{enumerate}[{(1)}]

  \item  A \emph{left ideal} of $V$ is a subsuperspace $I$ satisfying $u_s I \subseteq I$ for all $u \in V$ and $s \in \mathbb{Z}$.

  \item  An \emph{ideal} of $V$ is a left ideal $I$ with $u_n v \in I$ for all $u \in I, v \in V$  and $ n \in \mathbb{Z}$.

  \item The vertex superalgebra $V$ is \emph{irreducible} if it has no nonzero proper left ideals.

  \item The vertex superalgebra $V$ is \emph{simple} if it has no nonzero proper ideals. 
  
 \end{enumerate}
\end{definition}

\begin{remark}
	
By definition, every irreducible vertex superalgebra is simple. 
For $\frac{1}{2}\mathbb{N}$- or $\mathbb{N}$-graded vertex operator superalgebras, 
irreducibility and simplicity are equivalent. 
However, this equivalence is not universal: 
there exist simple vertex superalgebras that are not irreducible \cite{DRY2}.

\end{remark}

\begin{proposition} ([LL])
Let $V$ be a vertex superalgebra. Then $I$ is an ideal of $V$ if and only if $I$ is a $\mathcal{D}$-stable left ideal (i.e., $\mathcal{D} I \subset I$).
\end{proposition}


\subsection{Examples}

The following examples of vertex superalgebras will be useful later.

\begin{example} \label{exam1}
	Let $A$ be a commutative associative superalgebra with identity element $1_A$.  
	That is, for any homogeneous elements $a, b \in A$, we have $ab = (-1)^{|a||b|}ba$.  
	Let $\partial$ be an even superderivation of $A$, i.e., for any $a, b \in A$, we have $\partial(ab) = \partial(a)b + a\partial(b)$.  
	In this context, the pair $(A, \partial)$ is called a commutative differential superalgebra.  
	For $a, b \in A$, we define  
	\[Y^{(A,\partial)}(a,z)b = (e^{z\partial}a)b = \sum_{n=0}^{\infty}\frac{1}{n!} (\partial^n a) b  z^n.\]  
	Then $(A, Y^{(A,\partial)}( \ , z), 1_A)$ forms a commutative vertex superalgebra.  
	If there is no ambiguity, we may use $(A, \partial)$ to denote the vertex superalgebra.
	
	Let $\mathfrak{h} = \mathfrak{h}_{\bar{0}} \oplus \mathfrak{h}_{\bar{1}}$ be a vector superspace.  
	Let $\mathfrak{h} \otimes t^{-1}\mathbb{C}[t^{-1}]$ be the commutative Lie superalgebra with  
	even part $\mathfrak{h}_{\bar{0}} \otimes t^{-1}\mathbb{C}[t^{-1}]$ and odd part $\mathfrak{h}_{\bar{1}} \otimes t^{-1}\mathbb{C}[t^{-1}]$.  
	For simplicity, we use $h(-n)$ to denote $h \otimes t^{-n}$ for $h \in \mathfrak{h}$ and $n > 0$.  
	Let $\mathcal{F}(\mathfrak{h}) = \mathcal{U}(\mathfrak{h} \otimes t^{-1}\mathbb{C}[t^{-1}])$  
	be the universal enveloping algebra of the commutative Lie superalgebra $\mathfrak{h} \otimes t^{-1}\mathbb{C}[t^{-1}]$.  
	Then $\mathcal{F}(\mathfrak{h})$ is a commutative associative superalgebra.  
	Let $\partial$ be the even derivation of $\mathcal{F}(\mathfrak{h})$ uniquely determined by  
	$\partial(h(-n)) = n h(-n-1)$ for $h \in \mathfrak{h}$ and $n > 0$.  
	The pair $(\mathcal{F}(\mathfrak{h}), \partial)$ forms a free commutative differential superalgebra.  
	In particular, it is naturally a commutative vertex superalgebra.
	
	In what follows, if $\mathfrak{g}$ is a Lie superalgebra, $\mathcal{F}(\mathfrak{g})$  
	denotes the commutative vertex superalgebra associated with the underlying vector superspace of $\mathfrak{g}$.
	
\end{example}

\begin{example} \label{Ex-aff}  (\cite{K})
	Let $\mathfrak{g}$ be a finite dimensional Lie superalgebra with an even supersymmetric invariant bilinear form $(\ , )$.
	Consider the Affine Lie superalgebra defined by
	$$\widetilde{\mathfrak{g}}=\mathfrak{g} \otimes \C[t, t^{-1}] \oplus \C K,$$
	with Lie brackets given by:
	$$[x(m), y(n)]=[x,y](m+n)+m\delta_{m+n, 0}(x, y)K \ \ \text{and} \ \ [K, \ \widetilde{\mathfrak{g}}]=0,$$
	for $ x, y \in \mathfrak{g}$ and $ m ,n \in \Z,$  where $x(m)$ denotes $x\otimes t^m$.
	
	Given a complex number $k$, 
	let $\mathfrak{g}[t]$ act trivially on $\C$ and let $K$ act on $\C$ as multiplication by $k$, making $\C$ a $\mathfrak{g}[t] \oplus \C K$-module.
	We form the induced module
	$$V_{\widetilde{\mathfrak{g}}}(k, 0)=\mathcal{U}(\widetilde{\mathfrak{g}}) \otimes_{\mathfrak{g}[t] \oplus \C K} \C.$$
	Here and below,  $\mathcal{U}(\widetilde{\mathfrak{g}})$ denotes  the universal enveloping algebra of the Lie superalgebra $\widetilde{\mathfrak{g}}$.
    For convenience, set ${\bf 1}= 1 \otimes 1 \in V_{\widetilde{\mathfrak{g}}}(k, 0)$.
	Then $V_{\widetilde{\mathfrak{g}}}(k, 0)$ admits a unique vertex superalgebra structure satisfying
	$$Y(x(-1){\bf 1}, z) = \sum_{n \in \Z} x(n)z^{-n-1},$$
	for all $x \in \mathfrak{g}$.
	Note that $V_{\widetilde{\mathfrak{g}}}(k, 0) $ is an $\N$-graded vertex superalgebra,
	with $\text{deg}\ x(-1){\bf 1}=1$ for any $x \in \mathfrak{g}$.
	
\end{example}

\begin{example} \label{Ex-NS}  (\cite{K,Li3})
	Let $NS$ be the Neveu-Schwarz Lie superalgebra
	\begin{align*}
		NS=\big(\oplus_{m\in \Z}\C L(m) \big) \bigoplus  \big( \oplus_{n\in \Z}\C G(n+\frac{1}{2}) \big)  \bigoplus \C C,
	\end{align*}
	with the following commutation relations:
	\begin{align*}
		&[L(m), L(n)]=(m-n)L(m+n)+\frac{m^3-m}{12}\delta_{m+n, 0}C,\\
		&[L(m), G(n+\frac{1}{2})]=(\frac{m}{2}-n-\frac{1}{2})G(m+n+\frac{1}{2}),\\
		&[G(m+\frac{1}{2}), G(n-\frac{1}{2})]_+=2L(m+n)+\frac{1}{3}m(m+1)\delta_{m+n, 0}C,\\
		&[NS, C]=0.
	\end{align*}
	Let
	
	$$NS_+=\bigoplus_{n \geq 1} \big( \C L(n) \oplus \C G( n - \frac{1}{2}) \big), \ \ \text{and} \ \ NS_0=\C L(0)\oplus \C C.$$
	Then $NS_+ \oplus NS_0$ is a Lie subalgebra of $NS$. 
	For any $c \in C$, let $\C$ be the $(NS_+ \oplus NS_0)$-module 
	such that the actions of $ NS_+ \oplus \C L(0)$ on $\C$ are trivial, and the action of $C$ on $\C$ is multiplication by the scalar $c$.
	We now consider the induced module 
	$$V_{NS}(c, 0)=\mathcal{U}(NS) \otimes_{ (NS_+ \oplus NS_0)} \C.$$ 
	For convenience, we set ${\bf 1}= 1 \otimes 1 \in V_{NS}(c, 0)$.
	We further let $\widetilde{V}_{NS}(c, 0)=V_{NS}(c, 0)/\langle G(-\frac{1}{2}) {\bf 1} \rangle$,
	where $\langle G(-\frac{1}{2}){\bf 1} \rangle$ is the submodule generated by $G(-\frac{1}{2}){\bf 1}$. 
	Then $\widetilde{V}_{NS}(c, 0)$ forms a $\frac{1}{2}\N$-graded vertex operator superalgebra.
	This superalgebra is generated by the even element $L(-2) {\bf 1}$ of degree $2$ and the odd element $G(-\frac{3}{2}){\bf 1}$ of degree $\frac{3}{2}$. 
	The corresponding vertex operators are
	$$Y(L(-2) {\bf 1} , z)=\sum_{n \in \Z} L(n)z^{-n-2},$$ 
	and 
	$$Y(G(-\frac{3}{2}){\bf 1}, z)= \sum_{n \in \Z} G(n+\frac{1}{2})z^{-n-2}.$$
\end{example}


\section{$Y(z)$-injectivity for simple vertex superalgebras}

\begin{definition}
	A vertex superalgebra $V$ is said to be \emph{$Y(z)$-injective} if the linear map
	\[
	Y(z) \colon V \otimes V \to V(\!(z)\!) , \quad 
	u \otimes v \mapsto Y(u, z)v \quad \text{for } u, v \in V,
	\]
	is injective.
\end{definition}

\begin{remark}
Similar to the vertex algebra case,  $Y(z)$-injective vertex superalgebras possess many excellent properties.
For example: For any such vertex superalgebra $V$ and finite subgroup $G \leq \operatorname{Aut}(V)$, 
every irreducible representation of $G$ appears in $V$ 
(established analogously to the vertex algebra case in \cite[Proposition 4.2]{DRY2}). 
\end{remark}

Since the tensor product of vector spaces is left exact, the following property holds immediately.
\begin{lemma} \label{sub-Y(z)}
Let $V$ be a $Y(z)$-injective vertex superalgebra, and $U \subseteq V$ be a vertex subsuperalgebra.
Then $U$ is also $Y(z)$-injective.
\end{lemma}

In the remainder of this section, we shall prove the $Y(z)$-injectivity of countable-dimensional simple vertex superalgebras.
 
Let $V$ be a vertex superalgebra, 
and let $A(V, \mathcal{D})$ denote the associative subalgebra of $\text{End}(V)$
generated by the operators $\mathcal{D}$ and $u_n$ for $u \in V$ and $n \in \Z$.
In the following Lemma \ref{simple-AVD},  we treat  $A(V, \mathcal{D})$ as an ordinary algebra (not a superalgebra),
So the underlying vector space $|V|$ carries the structure of an $A(V, \mathcal{D})$-module.
Moreover, an $A(V, \mathcal{D})$-submodule $I \subseteq V$ is an ideal of $V$ if and only if $I$ is $\sigma_V$-stable.

\begin{lemma} \label{simple-AVD}
A vertex superalgebra $V$ is simple if and only if $|V|$ is a simple $A(V, \mathcal{D})$-module. 	
\end{lemma}

\begin{proof}

Assume that $V$ is a simple vertex superalgebra, but $|V|$ is not a simple $A(V, \mathcal{D})$-module.
Then there exists a nonzero proper $A(V, \mathcal{D})$-submodule $I$ of $|V|$. 
Clearly, both $I \cap \sigma_V(I)$ and $I + \sigma_V(I)$ are $\sigma_V$-stable $A(V, \mathcal{D})$-submodules;
consequently, they form ideals of $V$.
The simplicity of $V$ implies $I \cap \sigma_V(I) = 0$ and $I + \sigma_V(I) = V$ (i.e. $I \oplus \sigma_V(I)=V$).
Thus we define a linear isomorphism $f: V \to V$ by
\begin{align}\label{eq3-1}
	f(x) = x, \quad f(\sigma_V(x)) = -\sigma_V(x) \quad \text{for} \quad x \in I.
\end{align}
As $I + \sigma_V(I) = V$, we obtain the decompositions:
\begin{align}\label{eq3-2}
V_{\bar 0}=\{x + \sigma_V(x) ~ | ~ x\in I \} \ \ \ \text{and} \ \ \ V_{\bar 1}=\{x - \sigma_V(x) ~ | ~ x\in I \}.
\end{align}
From (3.3) we deduce $f V_{\bar 0} \subseteq V_{\bar 1}$ and $f V_{\bar 1} \subseteq V_{\bar 0}$.

We claim that \( fY(u, z)v = Y(u, z)fv \) for any \( u, v \in V \). 
To see this, observe that:
\begin{align*}
	\text{if } v \in I,         & \quad Y(u, z)fv = Y(u, z)v = fY(u, z)v, \\
	\text{if } v \in \sigma_V(I), & \quad Y(u, z)fv = -Y(u, z)v = fY(u, z)v.
\end{align*}


Since $I$ and $\sigma_V (I)$ are $\mathcal{D}$-stable, \eqref{eq3-1} implies $f \mathcal{D} = \mathcal{D} f$. 
For homogeneous $u, v \in V_{\bar 0}$, we compute:
\begin{align*}
	Y(fu, z)fv  
	&= fY(fu, z)v \\
	&= (-1)^{|fu||v|} f e^{z\mathcal{D}} Y(v, -z) fu \\
	&= f e^{z\mathcal{D}} f Y(v, -z) u \\
	&= f^2 e^{z\mathcal{D}} Y(v, -z) u \\
	&= f^2 Y(u, z)v,
\end{align*}
and
\begin{align*}
	Y(fu, z)fv  
	&= (-1)^{|fu||fv|} e^{z\mathcal{D}} Y(fv, -z) fu \\
	&= -e^{z\mathcal{D}} Y(fv, -z) fu \\
	&= -e^{z\mathcal{D}} f Y(fv, -z) u \\
	&= -f e^{z\mathcal{D}} Y(fv, -z) u \\
	&= -f e^{z\mathcal{D}} (-1)^{|fv||u|} e^{-z\mathcal{D}} Y(u, z) fv \\
	&= -f Y(u, z) fv \\
	&= -f^2 Y(u, z)v.
\end{align*}
As $f$ is a linear isomorphism, we have $Y(u, z) v = 0$ for all $u, v \in V_{\bar 0}$, 
contradicting $Y(\mathbf{1}, z) = \operatorname{id}$. 
Therefore, $|V|$ is a simple $A(V, \mathcal{D})$-module.
The converse is trivial, completing the proof.
 
\end{proof}

\begin{theorem} \label{S-YS}
If $V$ is a simple vertex superalgebra of countable dimension, then the linear map $Y(z)$ defined above is injective.
\end{theorem}

\begin{proof}
The proof is now similar to that of  \cite[Proposition 4.3]{DRY2}.	
Suppose that the linear map $Y(z)$ is not injective.
Then there exists  a nonzero vector
$v^1 \otimes w^1 + \cdots + v^s \otimes w^s$
in the kernel of $Y(z)$,
where $s$ is a positive integer, $v^1, \cdots, v^s$ are linearly independent,
and $w^1 , \cdots , w^s$ are nonzero.
That is, we have $$Y(v^1, z)w^1+ \cdots + Y(v^s,z)w^s=0.$$
By weak associativity, for any $u \in V$, there exists some $k \in \N$ such that
\begin{align*}
& (z+z_0)^k(Y(Y(u, z_0)v^1, z)w^1 + \cdots + Y(Y(u, z_0)v^s, z)w^s)\\
&=(z_0+z)^k(Y(u, z_0+z)Y(v^1,z)w^1 + \cdots + Y(u, z_0+z)Y(v^s,z)w^s)\\
&=0,
\end{align*}
which implies that
\begin{align*}
Y(Y(u, z_0)v^1, z)w^1 + \cdots + Y(Y(u, z_0)v^s, z)w^s=0.
\end{align*}
On the other hand, we have
\begin{align*}
& Y(\mathcal{D} v^1, z)w^1+ \cdots + Y(\mathcal{D} v^s,z)w^s \\
& =\frac{d}{dz}(Y(v^1, z)w^1+ \cdots + Y(v^s,z)w^s)\\
&=0.
\end{align*}
Therefore, for any $a \in A(V, \mathcal{D})$, we have
$$Y(av^1, z)w^1+ \cdots + Y(av^s,z)w^s=0.$$
Since $|V|$ is an irreducible $ A(V, \mathcal{D})$-module (see Lemma \ref{simple-AVD}), and  $v^1, \cdots, v^s$ are linearly independent,
by Jacobson density theorem there exists $a \in  A(V, \mathcal{D})$ such that
$av^1={\bf 1}$ and $av^i=0$ for any $i \neq 1$.
It follows that $Y({\bf 1}, z)w^1=0$, which is a contradiction.
Hence $Y(z)$ is injective and the proof is complete.
\end{proof}

\section{$Y(z)$-injectivity for  vertex superalgebras with PBW basis}

In this section, we will show that every $\frac{1}{T}\mathbb{N}$-graded vertex superalgebra with a PBW basis is $Y(z)$-injective. 
Our approach is motivated by the non-degeneracy arguments developed for quantum vertex algebras with PBW bases in \cite{Li1}.

Let $T$ be a positive integer, and let $V$ be a $\frac{1}{T}\mathbb{N}$-graded vertex superalgebra. Assume that $V$ is strongly generated by a $\frac{1}{T}\mathbb{N}$-graded subsuperspace $U$. For $p \in \frac{1}{T}\mathbb{N}$, let $E_p(V)$ denote the linear subsuperspace of $V$ spanned by vectors of the form  
\[
u^1_{-n_1} \cdots u^r_{-n_r} \mathbf{1},
\]  
where $r \geq 0$, $n_i \geq 1$, and $u^1, \dots, u^r$ are $(\Z_2 \times \frac{1}{T}\mathbb{N})$-homogeneous elements of $U$ satisfying  
\[
\deg u^1 + \cdots + \deg u^r \leq p.
\]  
Similar to \cite{A, Li2, Li3}, we have the following statements:
\begin{enumerate}[{(1)}]
\item ${\bf 1} \in E_0(V)$;
\item $E_p(V) \subset E_q(V)$ for \ $0 \leq p < q$;
\item $V=\cup_{p \in \frac{1}{T}\N} E_p(V)$;
\item $\mathcal{D}E_p(V) \subset E_p(V)$ \ for any $p$;
\item $u_nE_q(V) \subset E_{p+q}(V)$ \ for $u \in E_p(V)$, $n \in \Z$;
\item $u_nE_q(V) \subset E_{p+q-1}(V)$ \ for $u \in E_p(V)$, $n \in \N$.
\end{enumerate}
Define 
$$\text{gr}_E(V)=\bigoplus_{p \in \frac{1}{T}\N}E_p(V)/E_{p-\frac{1}{T}}(V).$$
Here and below,  $E_{n}(V)=0 \ \text{if} \ n<0$.

We note that $\text{gr}_E(V)$ inherits a natural superspace structure from $V$.
It follows from (1) to (6) above that 
$\text{gr}_E(V)$ forms a commutative  vertex superalgebra with the vacuum vector ${\bf 1}+E_{-\frac{1}{T}}(V)$,
whoses vertex operator map is uniquely determined by the n-products:
\[
\bigl(u + E_{p - \frac{1}{T}}(V)\bigr)_n \bigl(v + E_{q - \frac{1}{T}}(V)\bigr) = u_n v + E_{p + q - \frac{1}{T}}(V)
\]
for $u, v \in V$, $p,q \in \frac{1}{T}\N$, and $n \in \Z$.

\begin{lemma}\label{Yz-1}
Let $V$ be a $\frac{1}{T}\N$-graded vertex superalgebra. 
Assume that  $\text{gr}_E(V)$ is a $Y(z)$-injective vertex superalgebra. 
Then $V$ is also $Y(z)$-injective. 
\end{lemma}
\begin{proof}
For  each $p \in  \frac{1}{T}\N$ , 
let $L_p$ be a complement of the $E_{p-\frac{1}{T}}(V)$ in $E_{p}(V)$, so that
 $$E_p(V)=E_{p-\frac{1}{T}}(V) \oplus L_p \  \ \ \  \text{and} \ \ \ \ V= \oplus_{p \in \frac{1}{T}\N} L_p.$$

Assume that the map $Y(z)$ is not injective.
Then there exists a nonzero element $u^1\otimes v^1 + \cdots + u^n \otimes v^n \in \text{Ker} (Y(z)) $ for some positive integer $n$, 
where $u_1, u_2, \cdots, u_n$ are linearly independent elements in subspaces $L_{p_1}, L_{p_2}, \cdots, L_{p_n}$ 
with  indices  $p_1 \geq p_2 \geq \cdots \geq p_n \geq 0 $, 
and $v_1, v_2, \cdots, v_n$ are nonzero elements in subspaces $L_{q_1}, L_{q_2}, \cdots, L_{q_n}$ 
with  indices  $q_1  \geq q_2 \geq \cdots \geq q_n \geq 0$.
The construction of $L_p$ and the selection of $u_i \in L_{p_i}$ and $v_i \in L_{q_i}$ ensure that 
$$(u^1 +E_{p_1-\frac{1}{T}}(V)) \otimes (v^1 +E_{q_1-\frac{1}{T}}(V)) + \cdots +(u^n +E_{p_n-\frac{1}{T}}(V)) \otimes (v^n +E_{q_n-\frac{1}{T}}(V))$$
is a nonzero element in $\text{gr}_E(V)$.
On the other hand, since $u^1\otimes v^1 + \cdots + u^n \otimes v^n \in \text{Ker} (Y(z))$, we have
\begin{align*}
& Y_{\text{gr}_E(V)}(u^1 +E_{p_1-\frac{1}{T}}(V), z)(v^1 +E_{q_1-\frac{1}{T}}(V))+ \cdots \\
& + Y_{\text{gr}_E(V)}(u^n +E_{p_n-\frac{1}{T}}(V),z)(v^n +E_{q_n-\frac{1}{T}}(V))=0, 
\end{align*}
which contradicts the fact that $\text{gr}_E(V)$ is $Y(z)$-injective.
Therefore, the linear map $Y(z)$ is injective,  completing the proof.

\end{proof}

For convenience, we adopt the following definition.
\begin{definition}
	Let $V$ be a $\frac{1}{T}\N$-graded vertex superalgebra. 
	We say that $V$ {\em  admits a PBW basis } if there exists a vector superspace $\mathfrak{h}$ 
	such that $\text{gr}_E(V)$ is isomorphic to $(\mathcal{F}(\mathfrak{h}), \partial)$ as commutative vertex superalgebras.
\end{definition}

\begin{theorem} \label{main-Y(z)}
Every vertex superalgebra admitting a PBW basis is $Y(z)$-injective.
\end{theorem}

\begin{proof}

Lemma \ref{Yz-1} reduces the proof to showing that $(\mathcal{F}(\mathfrak{h}), \partial)$ is $Y(z)$-injective for any vector superspace $\mathfrak{h}$.

\medskip\noindent
\textbf{Case 1: Finite-dimensional $\mathfrak{h}$.} 
Assume $\mathfrak{h}$ is finite-dimensional. By Lemma \ref{sub-Y(z)}, 
it suffices to embed $\mathcal{F}(\mathfrak{h})$ as a vertex subsuperalgebra into a simple vertex superalgebra of countable dimension.

Choose a basis  $\{e_1, e_2, \cdots, e_s\}$ for $\mathfrak{h}_{\bar 0}$ and a basis  $\{f_1, f_2, \cdots, f_t\}$ for $\mathfrak{h}_{\bar 1}$.
Let $\overline{\mathfrak{h}}_{\bar 0}$ be the vector space with a basis $\{\bar{e}_1, \bar{e}_2,  \cdots, \bar{e}_s\}$,
and let $\overline{{\mathfrak{h}}}_{\bar 1}$ be the vector space with a basis $\{\bar{f}_1, \bar{f}_2,  \cdots, \bar{f}_t\}.$
Construct the commutative Lie superalgebra
$H=\mathfrak{h}_{\bar 0} \oplus \overline{\mathfrak{h}}_{\bar 0} \oplus \mathfrak{h}_{\bar 1} \oplus \overline{{\mathfrak{h}}}_{\bar 1}$
with even part $H_{\bar 0}= \mathfrak{h}_{\bar 0} \oplus \overline{\mathfrak{h}}_{\bar 0}$
and odd part $H_{\bar 1}= \mathfrak{h}_{\bar 1} \oplus  \overline{{\mathfrak{h}}}_{\bar 1}$.
Equip $H$ with a nondegenerate even supersymmetric bilinear form $( \ , )$:
$$(H_{\bar 0}, \ H_{\bar 1})=(H_{\bar 1}, \ H_{\bar 0})=0,$$
$$(e_i, \bar{e}_j)=(\bar{e}_j, e_i)=\delta_{i,j}, \ \ \ (e_i, e_j)=(\bar{e}_i, \bar{e}_j)=0,$$
$$(f_i, \bar{f}_j)=-(\bar{f}_j, f_i)=\delta_{i,j}, \ \ \ (f_i, f_j)=(\bar{f}_i, \bar{f}_j)=0,$$
for any $i, j$, where $\delta_{i,j}$ is the Kronecker delta.

Since the bilinear form $(\ , )$ is nondegenerate,
the Heisenberg vertex superalgebra  $V_{\widetilde{H}}(1,0)$ constructed in Example \ref{Ex-aff} is simple (see, for example, \cite{LL,K}).
By Theorem \ref{S-YS}, this simplicity implies $Y(z)$-injectivity of $V_{\widetilde{H}}(1,0)$.
Let $V(\mathfrak{h})$ be the vertex subsuperalgebra of $V_{\widetilde{H}}(1,0)$ generated by $h(-1){\bf 1}$, for $h \in \mathfrak{h}$.
The orthogonality  condition $(\mathfrak{h}, \mathfrak{h}) \equiv 0$ forces $V(\mathfrak{h})$ to be a commutative vertex superalgebra.
Furthermore, it is easy to see that the vertex superalgebra $V(\mathfrak{h})$ and $(\mathcal{F}(\mathfrak{h}), \partial)$ are isomorphic.
Therefore,  $\mathcal{F}(\mathfrak{h})$ is $Y(z)$-injective when $\mathfrak{h}$ is a finite-dimensional vector superspace. 

\medskip\noindent
\textbf{Case 2: Arbitrary-dimensional $\mathfrak{h}$.} 
We now establish $Y(z)$-injectivity for $\mathfrak{h}$ of arbitrary dimension. Suppose $\sum_{i=1}^n u^i \otimes v^i \in \ker Y(z) \subseteq \mathcal{F}(\mathfrak{h}) \otimes \mathcal{F}(\mathfrak{h})$.

There exists a finite-dimensional supersubspace $W \subseteq \mathfrak{h}$ such that all $u^i, v^j$ lie in the vertex subsuperalgebra $U \subseteq \mathcal{F}(\mathfrak{h})$ generated by $\{w(-1) \mid w \in W\}$. 
Since $U \cong \mathcal{F}(W)$ and $W$ is finite-dimensional, Case 1 implies $U$ is $Y(z)$-injective. Hence $\sum_{i=1}^n u^i \otimes v^i = 0$ in $U \otimes U$, and consequently in $\mathcal{F}(\mathfrak{h}) \otimes \mathcal{F}(\mathfrak{h})$. This proves $Y(z)$-injectivity of $\mathcal{F}(\mathfrak{h})$.

\medskip\noindent
The conclusion follows from Cases 1 and 2.

\end{proof}

As a direct application of Theorem \ref{main-Y(z)}, we establish the $Y(z)$-injectivity for the following classes of vertex superalgebras:
\begin{enumerate}[(i)]
	\item Tensor products of those admitting PBW bases,
	\item Affine vertex superalgebras,
	\item Neveu-Schwarz vertex superalgebras.
\end{enumerate}

\begin{corollary}
	
Let $V$ and $U$ be $\frac{1}{T}\mathbb{N}$-graded vertex superalgebras admitting PBW bases. 
Then the tensor product vertex superalgebra $V \otimes U$ also admits a PBW basis. Consequently, $V \otimes U$ is $Y(z)$-injective.
\end{corollary}

\begin{proof}
Note that the tensor product vertex superalgebra $V \otimes U$ is $\frac{1}{T}\mathbb{N}$-graded with
$$(V \otimes U)_n = \bigoplus_{i+j=n} V_i \otimes U_j$$
for any $n \in \frac{1}{T}\mathbb{N}$. Assume $V$ is strongly generated by $A \subseteq V$, and $U$ by $B \subseteq U$. Suppose further that $\operatorname{gr}_E(V) \cong \mathcal{F}(\mathfrak{h})$ and $\operatorname{gr}_E(U) \cong \mathcal{F}(\mathfrak{n})$ for some vector superspaces $\mathfrak{h}$ and $\mathfrak{n}$.
	
Then $V \otimes U$ is strongly generated by  $A \otimes {\bf 1} + {\bf 1} \otimes B$.
From the definition of filtration, we immediately obtain
$$E_n(V \otimes U) = \sum_{i+j=n} E_i(V) \otimes E_j(U)$$
for any $n \in \frac{1}{T}\mathbb{N}$. Therefore, we have the following isomorphism of commutative vertex superalgebras:
$$\operatorname{gr}_E(V \otimes U) \cong \operatorname{gr}_E(V) \otimes \operatorname{gr}_E(U) \cong \mathcal{F}(\mathfrak{h}) \otimes \mathcal{F}(\mathfrak{n}) \cong \mathcal{F}(\mathfrak{h} \oplus \mathfrak{n}).$$
Consequently, $V \otimes U$ admits a PBW basis. This completes the proof.
	
\end{proof}

\begin{remark}
It is shown in \cite{Li2} that the tensor product of nondegenerate nonlocal vertex algebras remains nondegenerate.
However, without the additional assumption of a PBW basis, the tensor product typically fails to preserve the $Y(z)$-injectivity.

\end{remark}

\begin{corollary}
Let $\mathfrak{g}$ be a finite-dimensional Lie superalgebra equipped  with an even supersymmetric invariant bilinear form $(\ , )$.
For any complex number $k$, the affine vertex superalgebra $V_{\widetilde{\mathfrak{g}}}(k, 0)$ constructed  in Example \ref{Ex-aff} is $Y(z)$-injective. 
\end{corollary}
\begin{proof}
Let $U=\mathfrak{g} \otimes t^{-1}$.
Then $V_{\widetilde{\mathfrak{g}}}(k, 0)$ is strongly generated by $U$.
By definition, the subspace  $E_n(V_{\widetilde{\mathfrak{g}}}(k, 0))$ is spanned by vectors of the form  
$x_1(-m_1) \cdots x_r (-m_r){\bf 1},$
where  $0 \leq r \leq n$,  $x_1, \cdots, x_r \in \mathfrak{g}$,  and $m_1, \cdots, m_r \geq 1$.
The  quotient space $E_n(V_{\widetilde{\mathfrak{g}}}(k, 0)) / E_{n-1}(V_{\widetilde{\mathfrak{g}}}(k, 0))$ is then spanned by the following vectors 
\begin{align} \label{eq5-1}
x^{k_1}_1(-m_1) \cdots x^{k_s}_s(-m_s) y_1(-n_1) \cdots y_t(-n_t){\bf 1} + E_{n-1}(V_{\widetilde{\mathfrak{g}}}(k, 0)),
\end{align}
where $x_1, \cdots, x_s \in  \mathfrak{g}_0$, \ $y_1, \cdots, y_t \in \mathfrak{g}_1$, \  $m_1> \cdots > m_s >0$, \ $n_1> \cdots > n_t >0$, and
$k_1+ \cdots +k_s +t=n.$
By the PBW Theorem, these elements from (\ref{eq5-1}) are linearly independent.
This induces a vector superspace isomorphism:
$$\text{gr}_E(V_{\widetilde{\mathfrak{g}}}(k, 0)) \cong  \mathcal{U}(\mathfrak{g} \otimes t^{-1}\C[t^{-1}]) = \mathcal{F}(\mathfrak{g}). $$ 
Importantly, this isomorphism is in fact also a vertex algebra isomorphism, uniquely determined by the map sending
$x(-1){\bf 1} + E_0(V_{\widetilde{\mathfrak{g}}}(k, 0))$ to $x(-1)$ for any $x \in \mathfrak{g}$.
It follows from Theorem \ref{main-Y(z)} that the affine vertex superalgebra $V_{\widetilde{\mathfrak{g}}}(k, 0)$ is $Y(z)$-injective.
 This completes the proof.

\end{proof}

\begin{corollary}
	The Neveu-Schwarz vertex superalgebra $\widetilde{V}_{NS}(c, 0)$ constructed in Example \ref{Ex-NS}  is $Y(z)$-injective.
\end{corollary}

\begin{proof}	
Let $$NS_{\geq -1}= \bigoplus_{n \geq -1} (\C L(n) \oplus \C G(n+\frac{1}{2})).$$
It is easy to verify that both $NS_{\geq -1}$ and $NS_{\geq -1} \oplus \C C$ are subalgebras of $NS$.
Consider $\C$ as a $NS_{\geq -1} \oplus \C C$-module, where $C$ acts as the scalar $c$, and $NS_{\geq -1}$ acts trivially.
Form the induced module 
$$M(c, 0)=\mathcal{U}(NS) \otimes_{ ( NS_{\geq -1} \oplus \C C )} \C. $$
Observe that in $\widetilde{V}_{NS}(c, 0) $, we have $L(-1) \1= G(-\frac{1}{2} ) G(- \frac{1}{2}) \1 =0$.
Consequently, in $\widetilde{V}_{NS}(c, 0) $, it holds that $NS_{\geq -1} \1=0 $ and $C \1 =c \1.$
Therefore, $\widetilde{V}_{NS}(c, 0)$ and  $M(c, 0)$ are isomorphic as  $NS$-modules.
By PBW theorem, $\widetilde{V}_{NS}(c, 0)$  has a basis consisting of the vectors
$$L(-n_s) \cdots L(-n_1) G(-m_t-\frac{1}{2}) \cdots G(-m_1-\frac{1}{2})\1$$
where $s+t>0$,  $n_s \geq n_{s-1} \geq  \cdots  \geq n_1 \geq 2$, and $  m_t >m_{t-1} > \cdots > m_1 \geq 1.$ 

Let  $\mathfrak{h}$ be a $(1,1)$-dimensional vector superspace with even part $\mathfrak{h}_{\bar 0}= \C x$ and odd part $\mathfrak{h}_{\bar 1}= \C y.$
It follows that  $\text{gr}_E(\widetilde{V}_{NS}(c, 0))$  and $ \mathcal{F}(\mathfrak{h})$ 
are isomorphic as commutative vertex superalgebras.
By Theorem \ref{main-Y(z)}, the Neveu-Schwarz vertex superalgebra $\widetilde{V}_{NS}(c, 0)$ is $Y(z)$-injective.
This completes the proof.

\end{proof}

\section{Hopf action on vertex superalgebras}

\subsection{Hopf algebras}

From now on, $H$ stands for a Hopf algebra with a structural data $(H,\mu,\eta,\Delta,\epsilon,S),$
where the linear maps
$$\mu:H\otimes H\rightarrow H,~\eta:\mathbb{C}\rightarrow H,~\Delta:H\rightarrow H\otimes H,~\epsilon:H\rightarrow \mathbb{C},~S:H\rightarrow H$$
 are multiplication, unit, comultiplication, counit, and antipode, respectively.
 We adopt Sweedler notation for comultiplication: for $h \in H$, we write $\Delta(h)=\sum h_1 \otimes h_2.$

\begin{definition}
A Hopf algebra $H$ is called cocommutative if $\sum h_1 \otimes h_2 = \sum h_2 \otimes h_1$ for any $h \in H$.
\end{definition}

\begin{lemma} \label{gr-alg} \cite{M}
If $H$ is a finite-dimensional cocommutative Hopf algebra, then it is a group algebra.
\end{lemma}

A subspace $I$ of a Hopf algebra $H$ is called a Hopf ideal if it satisfies the following conditions:

\begin{enumerate}[{(1)}]
\item  $IH \subseteq I$ and $HI \subseteq I$.

\item  $\Delta(I) \subseteq H \otimes I + I \otimes H$ and $\epsilon(I) =0$.

\item  $S(I) \subseteq I$.
\end{enumerate}
A subspace $I$ of $H$ with properties (1) and (2) is called a bialgebra ideal of $H$.

\begin{lemma} (\cite{W}) \label{bi-ideal}
 Every bialgebra ideal of a finite-dimensional Hopf algebra is a Hopf ideal.
\end{lemma}

\begin{definition}
	Given an $H$-module $M$, we say that $M$ is an inner faithful $H$-module if $IM \neq 0$ for every nonzero Hopf ideal $I$ of $H$.
	
\end{definition}

\begin{definition}  \label{Hopf-VSA}
Given a Hopf algebra $H$ and a vertex superalgebra $V$,
we say that $H$ acts on $V$ (or that $V$ is an $H$-module vertex superalgebra)
if the following conditions hold:
\begin{enumerate}[{(1)}]
  \item $V$ is an $H$-module satisfying $H V_{\alpha} \subseteq V_{\alpha}$ for any $\alpha \in \Z_2$.

  \item $h {\bf 1}=\epsilon(h){\bf 1}$, for any  $ h\in H.$

  \item For any $h \in H, u, v \in V$, we have $h (Y(u,z)v)=\sum Y(h_1u,z)h_2v$.
 
\end{enumerate}
\end{definition}

The following Lemma comes directly from  the definition.

\begin{lemma} \label{le3-7}
	Let $V$ be an $H$-module vertex superalgebra. Then
	
	\begin{enumerate}[{(1)}]
		\item $V^H$ is a vertex subsuperaglebra of $V$.
		
		\item The actions of $H$ and $\mathcal{D}$ on $V$ commute.
		
		\item The actions of $H$ and $V^H$ on $V$ commute.
	\end{enumerate}
\end{lemma}

\begin{definition}
	An action of a Hopf algebra $H$ on a vertex superalgebra $V$ is inner faithful if no nonzero Hopf ideal of $H$ annihilates $V$.
	
\end{definition}

\begin{remark}
Analogous to the vertex algebra case [DRY2], for any $H$-module vertex superalgebra $V$, 
there exists a unique maximal Hopf ideal $I \subset H$ satisfying $I \cdot V = 0$. 
This yields a quotient Hopf algebra $H/I$ and makes $V$ an inner faithful $(H/I)$-module vertex superalgebra. 
Crucially, this reduction preserves the invariant subsuperalgebra: $V^H = V^{H/I}$.
\end{remark}

In analogy with the proof for the vertex algebra case given in \cite[ Proposition 3.11]{DRY2}, we have the following proposition.

\begin{proposition} \label{ex-inn-fai1}
	Let $V$ be a vertex superalgebra and let $G$ be an automorphism group of $V$.
	Then $V$ is an inner faithful $\C[G]$-module vertex superalgebra.
\end{proposition}

\subsection{Hopf actions on $Y(z)$-injective vertex superalgebras }

\begin{lemma} \label{thm-Hopf-ideal}
Let $H$ be a Hopf algebra, and let $V$ be an $H$-module vertex superalgebra such that $Y(z)$ is injective.
Let $K=\{h \in H \ | \ hv=0 \  \text{for all} \  v \in V  \}$ be the kernel of the action of $H$ on $V$.
Then $K$ is a bialgebra ideal of $H$. 
\end{lemma}
\begin{proof}
In the case that $V$ is a vertex algebra, the exactly same results were
obtained in [DRY2]. The same proof works here.
\end{proof}

\begin{lemma} \label{H-iso}
Let $H$ be a Hopf algebra, and let $V$ be an $H$-module vertex superalgebra such that $Y(z)$ is injective.	
Let
$$\tau: V \otimes V \to V \otimes V$$
be the linear map defined by 
$$\tau (u \otimes v) = (-1)^{|u| |v|} (v \otimes u)  \  \ \  \text{for \ homogeneous} \  \ u,  v \in V.$$
Then the linear map $\tau$ is an $H$-isomorphism.

\end{lemma}

\begin{proof}
This proof is essentially the same as the proof for the vertex algebra case in \cite[Theorem 5.9]{DRY2}.
Since $V$ is an $H$-module, $H$ acts on the coefficients of the formal series, endowing $V\{z\}$ with an $H$-module structure.
To continue the proof, we set
$$\mathcal{V}^0=\text{span}\{Y(u,z)v \ | \ u, v \in V \} \subseteq V\{z\},$$
and $$\mathcal{V}^1=\text{span}\{Y(u,-z)v \ | \ u, v \in V \} \subseteq V\{z\}.$$
As $V$ is an $H$-module vertex superalgebra, we can see that $\mathcal{V}^0$ and $\mathcal{V}^1$ are $H$-submodules of $V\{z\}$.
Given that the actions of $\mathcal{D}$ and $H$ on $V$ commute,
it can be deduced from Proposition \ref{Der}(2) that the map $e^{-z\mathcal{D}}: \mathcal{V}^0 \rightarrow \mathcal{V}^1$ is an $H$-isomorphism.

Since $Y(z)$ is injective,  the map $Y(z): V \otimes V \to \mathcal{V}^0$ is an $H$-isomorphism.
Similarly, it is easy to verify that the linear map
$$\widetilde{Y}(z): V \otimes V \to \mathcal{V}^1 \ \ \text{defined by } \ \ \widetilde{Y}(z) (u \otimes v)=Y(u,-z)v  \ \ \text{for} \ u, v \in V,$$
is also an $H$-isomorphism.
A straightforward calculation shows that $\tau =(\widetilde{Y}(z))^{-1}e^{-z\mathcal{D}}Y(z)$.
Therefore $\tau$ is an $H$-isomorphism.
The proof is complete.

\end{proof}

\begin{theorem} \label{thm-gr-alg}
Let $H$ be a finite-dimensional Hopf algebra.
Let $V$ be an inner faithful $H$-module vertex superalgebra such that $Y(z)$ is injctive.
Then $H \cong \C[G]$ as Hopf algebra for some finite automorphism group $G$ of $V$.
In paticular, $H$ must be a group algebra.
\end{theorem}
\begin{proof}
	
The proof follows arguments similar to those in \cite[Theorem 5.5]{DRY2}. 
Let $K$ denote the kernel of the $H$-action on $V$. By Lemma \ref{bi-ideal} and Lemma \ref{thm-Hopf-ideal}, $K$ is a Hopf ideal of $H$. 
Since $V$ is an inner faithful $H$-module, we must have $K = 0$.
Thus $V$ is a faithful $H$-module, and consequently, the tensor product $V \otimes V$ is a faithful $H \otimes H$-module. 	
	
Furthermore, Lemma \ref{H-iso} establishes that the linear map $\tau: V \otimes V \to V \otimes V$ defined by $\tau(u \otimes v) = (-1)^{|u||v|} v \otimes u$ for homogeneous elements $u, v \in V$ is an $H$-isomorphism. This implies the identity 
$\sum h_{(1)}v \otimes h_{(2)}u = \sum h_{(2)}v \otimes h_{(1)}u$
for all $h \in H$ and $u, v \in V$. It follows that 
$\sum h_{(1)} \otimes h_{(2)} = \sum h_{(2)} \otimes h_{(1)}$
for all $h \in H$, proving the cocommutativity of $H$. Hence, by Lemma \ref{gr-alg}, $H \cong \mathbb{C}[G]$ as Hopf algebras for some finite group $G$.
	
To complete the proof, we show that $G$ embeds into $\operatorname{Aut}(V)$.  
Since $\mathbb{C}[G]$ is a Hopf algebra with coproduct $\Delta(g) = g \otimes g$ and counit $\varepsilon(g) = 1$ for $g \in G$, Definition \ref{Hopf-VSA} implies  
	\[
	g Y(v, z) w = Y(g v, z) g w  \quad \forall v, w \in V,   \  \quad \text{and} \quad g \mathbf{1} = \mathbf{1}.
	\]  
As $V$ is a $\mathbb{C}[G]$-module, each $g \in G$ acts invertibly on $V$ (with inverse $g^{-1}$). 
Faithfulness implies that $g|_V = \operatorname{id}_V$ only when $g = 1_G$, Thus, the map $g \mapsto (v \mapsto gv)$ embeds $G$ into $\operatorname{Aut}(V)$.

\end{proof}

\end{document}